\documentclass[11pt, oneside]{amsart}
\usepackage{amsmath}
\usepackage{amsthm}
\usepackage{amssymb} 
\usepackage[mathscr]{euscript}
\usepackage{amsmath}
\usepackage{amsthm}
\usepackage{amssymb}
\usepackage{fancyhdr}
\usepackage{enumerate}
\usepackage{amscd}
\usepackage[all]{xy}
\usepackage{graphicx}
\usepackage{color}
\usepackage{hyperref}
\usepackage{comment}
\usepackage{thmtools}
\usepackage{cleveref}
\usepackage{thm-restate}
\usepackage{filecontents}
\usepackage[backend=biber]{biblatex}
\bibliography{AH MME Staticity}
\hypersetup{
    colorlinks,
    citecolor=black,
    filecolor=black,
    linkcolor=black,
    urlcolor=black
}
\usepackage{tikz}
\usepackage[all]{xy}
\usepackage{graphicx}
\usetikzlibrary{backgrounds}

\newtheorem*{theorem*}{Theorem}

\theoremstyle{plain}
\newtheorem{lemma}{Lemma}[section]

\newtheorem{proposition}[lemma]{Proposition}
\newtheorem{theorem}[lemma]{Theorem}

\numberwithin{equation}{section}

\theoremstyle{remark}

\theoremstyle{definition}
\newtheorem{remark}[lemma]{Remark}

\newtheorem{definition}[lemma]{Definition}

\newenvironment{enumeratei}
{\begin{enumerate}[\upshape (i)]}{\end{enumerate}}

\usepackage[tmargin=1in,bmargin=1in,lmargin=1in,rmargin=1in]{geometry}
\begin{document}
\title[Staticity of Asymptotically Hyperbolic Minimal Mass Extensions]{Staticity of Asymptotically Hyperbolic Minimal Mass Extensions}
\author{Daniel Martin}

\date{\today}

\begin{abstract}
	In this paper, we give a definition for the Bartnik mass of a domain whose extensions are asymptotically hyperbolic manifolds.  With this definition, we show that asymptotically hyperbolic admissible extensions of a domain that achieve the Bartnik mass must admit a static potential.  Given a non-static admissible extension of a domain, we are able to construct a one-parameter family of metrics that are close to the original metric, have smaller mass, share the same bound on the scalar curvature, and contain the domain isometrically.  
\end{abstract}

\maketitle

\section{\label{sec:level1}Introduction}

The ADM-mass is a well-studied geometric invariant defined at the infinity of an asymptotically flat manifold, which quantifies the mass of an isolated gravitational system.  
In an attempt to define a quantity that corresponds to the mass contribution coming from a bounded region with non-negative curvature and compact boundary, Bartnik introduced a quasi-local mass, known now as Bartnik mass \cite{Bartnik89}.  Bartnik first introduced a class of asymptotically flat manifolds, called admissible extensions, of such a given bounded region.  The Bartnik mass is then defined as the infimum of their corresponding ADM-masses.  Because of its close relation to the ADM-mass, the Bartnik mass shares some desirable properties with its global counterpart, such as non-negativity.  Along with this new definition, Bartnik made several conjectures regarding some of its characteristics.  Of particular interest is one that is the focus of this paper, which is the staticity of extensions that achieve the minimal ADM-mass, see \cite{HuiskenIlmanen01}.  That is, if there is an admissible extension of a domain that achieves the minimum ADM-mass, then it must admit a function called a static potential.  These functions make up the kernel of the formal $L^2$-adjoint of the linearized scalar curvature operator and determine the ability to deform the scalar curvature.  The study of Bartnik mass, and other notions of quasi-local mass, are far from complete - even in the asymptotically flat setting.  For instance, it is not yet known whether there always exists a minimal mass extension given any domain.  A nice discussion on the topic in the asymptotically flat setting can be found in \cite{Corvino2017}.  In this paper we define an appropriate analog of the Bartnik mass for bounded regions with compact boundary and scalar curvature bounded below that of hyperbolic space, $-n(n-1)$.  These regions will admit asymptotically hyperbolic extensions, and our main objective is to provide a proof of his staticity conjecture in this new setting.  

In \cite{Corvino2000}, Corvino proves that asymptotically flat extensions of a domain achieving the Bartnik mass must be static.  These extensions achieving the Bartnik mass are referred to as minimal mass extensions and are asymptotically flat manifolds with non-negative scalar curvature, no closed minimal surfaces separating the boundary from infinity, and contain the given domain isometrically.  The proof follows from his result that allows for scalar curvature deformation when no static potential exists on that region.  We state this result in the appendix for reference, see [\ref{proposition:bump}].  In the current work, we will employ similar tools and strategy to prove the staticity conjecture for asymptotically hyperbolic manifolds.

In this new setting, the class of admissible extensions will consist of asymptotically hyperbolic manifolds containing the given domain isometrically, scalar curvature bounded below by $-n(n-1)$, and contain no closed minimal surfaces separating the boundary from infinity, other than possibly on the boundary.  The analogous staticity conjecture in the asymptotically hyperbolic setting is the following, which is our main result:

\begin{theorem*}
	Let $(M^n,g)$, for $3 \le n \le 7$, be an asymptotically hyperbolic minimal mass extension of a domain $\Omega$.  Then $M \setminus \overline{\Omega}$ admits a static potential. 
\end{theorem*}

Another tool we will use is motivated by a technique of Andersson, Cai, Galloway \cite{Andersson2007} in their work on the positive mass theorem and rigidity statement for a class of conformally compact asymptotically hyperbolic manifolds.  We should mention that their definition of asymptotically hyperbolic, which we also adopt in this paper, is slightly stronger than other definitions used more recently in the literature, see \cite{Chrusciel2018, Huang2019, Pacheco2018}.  Since then the corresponding positive mass theorem and rigidity statement with weaker asymptotic assumptions have also been proven using different approaches \cite{Chrusciel2018,Huang2019}.  Andersson, Cai, and Galloway prove that if an asymptotically hyperbolic metric has scalar curvature bounded below by $-n(n-1)$, with strict inequality somewhere, they can produce another asymptotically hyperbolic metric with constant scalar curvature $-n(n-1)$ and strictly smaller mass.  In the following, we generate a \emph{family} of metrics with similar properties under a condition on the non-existence of a static potential instead of their assumption on the scalar curvature.  More specifically, we prove the following result, which is one of the main tools we use to prove our main theorem:

\begin{theorem*}
	Let $(M^n,g)$ be an asymptotically hyperbolic manifold with $R_g \ge -n(n-1)$.  If there exists a domain $\Omega$ such that no static potential exists on $M \setminus \overline{\Omega}$, then there is a family asymptotically hyperbolic metrics, $g_s$ for $0<s<1$, on $M$ such that $g_{s} \to g$ in $C^{k,\alpha}(M)$ as $s \to 0$ and 
	\begin{enumeratei}
		\item $\mu(g_{s}) < \mu(g)$,
		\item $R_{g_{s}} \ge -n(n-1)$,
		\item $g_{s} = g$ on a neighborhood of $\Omega$.  
	\end{enumeratei}
\end{theorem*}

Given the ability to generate such a family, our staticity result follows almost immediately via a contradiction proof.  That is, given a minimal mass extension for which no static potential exists on $M \setminus \overline{\Omega}$, we may produce a sequence of asymptotically hyperbolic manifolds with smaller mass.  Since every metric in this family has smaller mass, contains $\Omega$ isometrically, and has scalar curvature bounded below by $-n(n-1)$, they must all contain a closed minimal surface in the interior separating the boundary from infinity.  Otherwise, they would all be admissible extensions and would contradict the minimality property of the minimal mass extension.  Assuming that every such metric admits such a closed minimal surface, we can construct a sequence with a limiting closed minimal surface.  This closed minimal surface separates the boundary from infinity within the original manifold and allows us to arrive at our contradiction.  The dimension restriction on the main theorem is due to compactness results of Schoen-Simon in \cite{schoensimon} regarding the convergence of such a sequence of minimal surfaces.

\section{Background}

In the previous section, we mentioned that our definition of asymptotically hyperbolic is comparatively stronger than is more commonly used in the literature at the current moment.  In particular, we assume that our asymptotically hyperbolic manifolds are conformally compact.  Despite these stronger asymptotics, Dahl and Sakovich prove \cite{Dahl2015} that these conformally compact asymptotically hyperbolic manifolds are actually dense in the set of asymptotically hyperbolic manifolds that are physically reasonable.  For this reason, we continue to follow the definition given in \cite{Andersson2007}:
\begin{definition}\label{def:AH}
	A Riemannian manifold $(M^n, g)$ is \emph{asymptotically hyperbolic} if it is conformally compact with smooth conformal compactification $(\widetilde{M},\widetilde{g})$, and with conformal boundary $\partial_{\infty} \widetilde{M} = \mathbb{S}^{n-1}$ such that on a deleted neighborhood $(0,T) \times \mathbb{S}^{n-1}$ of $\partial_{\infty} \widetilde{M} = \{t = 0 \}$ the metric takes the form
	\[ g = \sinh^{-2}(t) \big( dt^2 + \mathring{h} + t^n\gamma + \mathcal{O}(t^{n+1})\big), \]
	where $h(t) = \mathring{h} + t^n \gamma + \mathcal{O}(t^{n+1})$ is a family of metrics on $\mathbb{S}^{n-1}$, depending smoothly on $t \in [0,T)$, $\mathring{h}$ is the round metric on $\mathbb{S}^{n-1}$, and $\gamma$ is a symmetric two-tensor on $\mathbb{S}^{n-1}$.
\end{definition}

\begin{remark}
	Note that in the above definition, the order notations $\mathcal{O}(t)$ and $o(t)$ throughout the paper are to be understood as $t \to 0$.  That is, at the conformal infinity of the manifold.
\end{remark}

Given this definition of asymptotically hyperbolic, there is also a corresponding notion of mass analogous to the ADM-mass in the asymptotically flat setting.  Similarly, it is also a geometric invariant within this class of metrics.  

\begin{definition}\label{def:mass}
	Let $(M^n,g)$ be an asymptotically hyperbolic manifold with $R_g \ge -n(n-1)$ and metric expression at infinity as in Definition \ref{def:AH}.  We define the \emph{mass} to be
	\[ m(g) = \int\limits_{\mathbb{S}^{n-1}} \mu(g) \, dV^{n-1}_{\mathring{h}}, \]
	where $\mu(g) = \mathrm{tr}_{\mathring{h}} \gamma$ is the \emph{mass aspect function}. 
\end{definition}

For simplicity, our results are all in terms of the mass aspect function, but can also be stated in terms of the mass defined previously. 

Next, we give our adaptation of Bartnik mass in the asymptotically hyperbolic setting.  

\begin{definition}\label{def:bmass}
	Let $(\Omega^n, g_{\omega})$ be a bounded domain with $R_{g_{\omega}} \geq - n(n-1)$.  Define $\mathcal{A}_{\Omega}$ to be the class \emph{admissible extensions} of $(\Omega, g_{\omega})$, which are $n$-dimensional asymptotically hyperbolic manifolds that contain $(\Omega, g_{\omega})$ isometrically, with scalar curvature bounded below by $-n(n-1)$, and contain no closed minimal surfaces separating the boundary from infinity other than possibly on the boundary.  
\end{definition}

\begin{remark}
	The final condition for admissibility, the so-called ``no horizon condition", is one we feel deserves closer inspection.  In the asymptotically flat setting, this condition is in place to avoid situations that would render the Bartnik mass trivially zero.  That is, if we were to allow such closed minimal surfaces, one could generate admissible extensions by gluing a Schwarzschild end with arbitrarily small mass to the domain in a suitable way.  A good discussion of this can be found in \cite{HuiskenIlmanen01,Corvino2017}. 
	
	However, in the present setting, while our definition again follows this same convention, it is not yet clear whether this is the optimal.  It appears plausible that one would run into similar issues by allowing the interior of the manifold to contain constant mean curvature $n-1$ surfaces that separate the boundary from infinity.  Similarly, we could foresee gluing an AdS-Schwarzschild end of arbitrarily small mass to our initial domain.  In any case, our results are adaptable to either definition.  
\end{remark}

\begin{definition}
	Define the \emph{Bartnik mass} of $(\Omega^n, g_{\omega})$ by
	\[ m_B(\Omega, g_{\omega}) = \inf\{m(g) : (M^n,g) \in \mathcal{A}_\Omega\}.\]
\end{definition}

Lastly, we discuss static potentials and some of the properties relevant to our discussion.  
\begin{definition}
	We say that a metric $g$ is \emph{static} if there exists a smooth non-trivial function $f$ satisfying the equation
	\[ L_g^*f = -(\Delta_gf)\,g + \mathrm{Hess}_gf - f\, \mathrm{Ric}_g = 0.\] 
	A function satisfying this equation is referred to as a \emph{static potential}.  Here $L_g^*$ is the formal $L^2$-adjoint to the linearized scalar curvature operator.  
\end{definition}

\begin{remark}\label{remark:csc}
	Manifolds admitting static potentials have constant scalar curvature.  
\end{remark}

\section{Proof of Main Results}

We begin by outlining the procedure we will use for constructing a family of asymptotically hyperbolic manifolds.  More specifically, we start with a given asymptotically hyperbolic manifold with scalar curvature bounded below by $-n(n-1)$ that contains a region whose exterior does not admit a static potential.  Using this presumption of non-existence of a static potential, we apply a result of Corvino \cite{Corvino2017}, Proposition \ref{proposition:bump}, to bump up the scalar curvature in this region.  This allows us to assume the scalar curvature is strictly larger than $-n(n-1)$ somewhere.  Since the scalar curvature is now non-constant, we generate a solution to the Yamabe equation for the prescribed constant $-n(n-1)$ scalar curvature.  With this solution we define a family of conformally related metrics that will change the asymptotics at infinity and cause the mass to decrease.  Lastly, we use a bump function to glue the given metric around the non-static region with the conformal metric at infinity and maintain all of the properties we need.  This will prove the following theorem:
\begin{theorem}\label{family}
	Let $(M^n,g)$ be an asymptotically hyperbolic manifold with $R_g \ge -n(n-1)$.  If there exists a domain $\Omega \subset M$ such that no static potential exists on $M \setminus \overline{\Omega}$, then there is a family asymptotically hyperbolic metrics, $g_s$ for $0 < s < 1$, on $M$ such that $g_{s} \to g$ in $C^{k,\alpha}(M)$ as $s \to 0$ and 
	\begin{enumeratei}
		\item $\mu(g_{s}) < \mu(g)$
		\item $R_{g_{s}} \ge -n(n-1)$
		\item $g_{s} = g$ on a neighborhood of $\Omega$.  
	\end{enumeratei}
\end{theorem}
\begin{proof}
	Let $(M^n,g)$ be an asymptotically hyperbolic manifold with $R_g \ge - n(n-1)$ containing a domain $\Omega$ for which no static potential exists on $M \setminus \overline{\Omega}$.  The proof begins with a reduction to the case that $R_g > -n(n-1)$ somewhere.  
	
	First, we show that we can bump the scalar curvature up locally if necessary.  Since no static potential exists on $M\setminus \overline{\Omega}$, there exists a precompact domain $V \subset M \setminus \overline{\Omega}$ for which $g$ does not admit a static potential.  Indeed, if every precompact set in $M \setminus \overline{\Omega}$ was static, an exhaustion of $M \setminus \overline{\Omega}$ could be created, along with a sequence of static potentials.  It follows that this sequence will converge to a static potential on $M \setminus \overline{\Omega}$ - for details see \cite[Section 2.1]{Corvino2017}.  Since $g$ is not static on $V$, by Lemma \ref{Lemma:asymptotics}, there exists a smooth metric close to $g$ in the $C^{k,\alpha}(\overline{V})$ sense with compact support strictly contained within $V$ and has scalar curvature that is strictly larger than $-n(n-1)$.  Note that this new metric on $M \setminus \overline{\Omega}$ is not static since the scalar curvature is not constant on $M \setminus \overline{\Omega}$.  Thus, we may assume without loss of generality that $R_g > -n(n-1)$ somewhere.
	
	Under the assumption that the scalar curvature inequality is strict somewhere on $M$, we now produce the family of metrics in the conclusion of the theorem.  The following analysis of the solution can also be found in \cite[Section 3.2.2]{Andersson2007}.  We start by studying solutions to the Yamabe equation for the prescribed scalar curvature $-n(n-1)$,
	\[ -\dfrac{4(n-1)}{n-2}\,\Delta_g u + R_g\,u + n(n-1)\, u^{\frac{n+2}{n-2}} = 0. \]
	There exists a unique positive solution $u$ on $M$ such that $u \to 1$ at the conformal infinity.  If we let $v = u - 1$ so that $v \to 0$ at the conformal infinity and define $\hat{R} = \dfrac{n-2}{4(n-1)} \,\big(R_g + n(n-1)\big)$, then the Yamabe equation becomes
	\[ -\Delta_g  v + nv + \hat{R} v = -\hat{R} - F(v) ,\]
	where 
	\[ F(v) = \dfrac{n(n-2)}{4} \left[ (1 + v)^{\frac{n+2}{n-2}} -1 - \dfrac{n+2}{n-2} \,v \right]. \]
	The analysis allows us to conclude that $v \le 0$ via the maximum principle with $v = v_{n}t^n + o(t^n)$ being smooth up to the conformal boundary.  Additionally, the function $v_n$ is a function defined only on $\mathbb{S}^{n-1}$.  Furthermore, when re-writing the Yamabe equation in $v$, we have
	\begin{align}\label{laplace}
		- \Delta_g v + nv = f, 
	\end{align}
	where $f = -\hat{R}(1+v) - F(v) = \mathcal{O}(t^{n+1})$.  Because $f \le 0$ and $f \not\equiv 0$ due to the scalar curvature observing the strict inequality, it follows that $v_{n}(x)< 0$, where $x \in \mathbb{S}^{n-1}$.  
	
	Next, define $h_s = u_s^{\frac{4}{n-2}}g$, where $u_s = 1 + sv$ and $v$ is a solution to (\ref*{laplace}) with asymptotics described above.  By the formula for scalar curvature under conformal change, we have
	\[ R_{h_s} = u_s^{-\frac{4}{n-2}}R_g - \frac{4(n-1)}{n-2}u_s^{-\frac{n+2}{n-2}}\Delta_gu_s \]
	Due to the fact that $-1 < v \le 0$ and $0 < s < 1$, we have $0 < u_s \le 1$ with $u_s^{-\frac{4}{n-2}} \geq 1$ and $u_s^{-\frac{n+2}{n-2}} \ge 1$.  Since $v_n < 0$ and $-\Delta_g u_s = -s(nv - f) = -s(nv_nt^n + o(t^{n}))$ we may choose a sufficiently small neighborhood of infinity such that $-\Delta_g u_s \ge 0$.  Lastly, these bounds imply that $R_{h_s} \ge R_g \ge - n(n-1)$ on a sufficiently small neighborhood of infinity, say $W_0 := (0,t_0) \times \mathbb{S}^{n-1}$.  Here $t_0 > 0$ is also chosen small enough so that $W_0$ does not intersect $\Omega$ and is not dependent on $s$.  
	
	Let $W_1 := (0,t_1) \times \mathbb{S}^{n-1} \supset W_0$ be an open set with $t_1> t_0$ chosen small enough so that $M \setminus \overline{W}_1 \supset \Omega$. Now, we define $g_s = (1-\varphi)g + \varphi\,h_s$, where $\varphi$ is smooth function on $M$ such that $\varphi \equiv 1$ on $W_0$, $\varphi \equiv 0$ on $M \setminus \overline{W}_1$, and non-increasing as $t$ increases.  Note that $\varphi$ can be chosen so $g_s \to g$ in $C^{k,\alpha}(M)$ by adjusting the neighborhoods $W_0$ and $W_1$ to control the derivatives as necessary.  By the definition of $h_s$, we have $h_s \to g$ as $s \to 0$ so that $g_s \to g$.  Since the metrics $g_s$ and $h_s$ are identical at infinity, these metrics are asymptotically hyperbolic with $\mu(g_s) = \mu(h_s) < \mu(g)$ by Lemma \ref{Lemma:asymptotics}.  
	
	Lastly, since we know $R_{g_s} = R_g \ge -n(n-1)$ on $M \setminus \overline{W}_1$ and $R_{g_s} \ge -n(n-1)$ on $W_0$, as shown above, it suffices to show that $R_{g_s} \ge - n(n-1)$ on $V_0: = W_1 \setminus \overline{W}_0$.  If $g_s$ was static on $V_0$, the scalar curvature would be constant, and thus would be greater than or equal to $-n(n-1)$ since it shares a boundary with $W_0$.  On the other hand, if $g_s$ was not static on $V_0$ we may again apply Lemma \ref{proposition:bump} to bump up the scalar curvature so that it exceeds $-n(n-1)$ and remains unchanged outside $V_0$.  Note that for the sake of matching the notation in Proposition \ref{proposition:bump}, we assume $V$ is a precompact domain containing $V_0$ that is disjoint from $\Omega$, which cannot be static when $V_0$ is not.  The metric $g$ and $V_0$ then determine a positive constant $\varepsilon_0$.  Let $k_s$ be a smooth non-negative function on $M$ with compact support in $\overline{V}_0$ such that $(R_{g_s} + n(n-1))^- \le k_s \le \mathrm{max}_{\overline{V}_0}(R_{g_s}+n(n-1))^{-}$.  When $g_s$ is sufficiently close to $g$ by choosing a small $s > 0$, we have $\Vert k_s\Vert_{C^{0,\alpha}} < \varepsilon_0$.  By the proposition, there exists a smooth metric $g_s + h$, where $h$ is supported in $\overline{V}_0$ such that $R_{g_s + h} = R_{g_s} + k_s \ge -n(n-1)$.  Note that $h$ also depends on $s$, but we avoid using an $s$ to index as to not conflict with the previously established $h_s$.  We rename $g_s = g_s +h$ if this procedure is necessary.  Thus, the metrics $g_s$ have the desired properties. 
\end{proof}

\begin{remark}
	Note that since our original manifold is not static, we know that the mass must be positive, due to the positive mass theorem and its rigidity statement.  Indeed if the mass was zero, then it would be isometric to hyperbolic space and thus static. 
\end{remark}

\begin{theorem}\label{theorem:main}
	Let $(M^n,g)$, for $3 \le n \le 7$, be an asymptotically hyperbolic minimal mass extension of a domain $\Omega$.  Then $M \setminus \overline{\Omega}$ admits a static potential. 
\end{theorem}

\begin{proof}
	Let $(M^n,g) \in \mathcal{A}_\Omega$ be a minimal mass extension of $\Omega$, and suppose that $M\setminus \overline{\Omega}$ does not admit a static potential.  Then by Theorem \ref{family}, there exists a family of asymptotically hyperbolic metrics $g_s \to g$ as $s \to 0$ such that $\mu(g_s) < \mu(g)$, $R_{g_s} \ge - n(n-1)$, and $g_s = g$ on $\Omega$.  To prove $(M, g_s) \in \mathcal{A}_\Omega$, it suffices to show that there are no closed minimal surfaces separating $\partial \Omega$ from infinity, as the other properties of admissible extensions are clearly satisfied.  
	
	First, recall that $g_s \to g$ as $s \to 0$ in $C^{k,\alpha}(M)$.  If every metric $g_s$ admits a closed minimal surface separating $\partial \Omega$ from infinity, choose a sequence $s_k \to 0$, and let $\Sigma_k$ be an aforementioned closed minimal surface with respect to the metric $g_{s_k}$.  Within the homology class of each surface we can choose a representative that is area-minimizing.  Thus, we may assume that these surfaces are stable at the outset.  Additionally, these surfaces must all be contained within a convex spherical barrier.  If not, we would be able to find a large sphere that touches the minimal surface tangentially, where we could compare mean curvatures and apply the comparison principle to arrive at a contradiction.  Since this family of metrics converges to the fixed metric $g$, the areas of $\Sigma_k$ are uniformly bounded by the area of the barrier.  Thus by the compactness results of Schoen and Simon \cite{schoensimon}, there is a subsequence of $\Sigma_k$ converging to a smooth minimal closed surface $\Sigma$ in $(M,g)$.  Since the metric is unchanged in a neighborhood of $\Omega$, we have the existence of a closed minimal surface separating the boundary from infinity - contradicting the assumption that $(M,g) \in \mathcal{A}_\Omega$.  Thus, $M \setminus \overline{\Omega}$ will admit a static potential. 
\end{proof}

\begin{remark}
	The remark following the definition of Bartnik mass mentions that an alternate to the no horizon condition may be more optimal.  Our proof would remain unchanged except for appealing to the compactness results of Schoen and Simon \cite{schoensimon}.  Given analogous results on the convergence of constant mean curvature surfaces, our result would easily adapt under an alternate definition regarding the non-existence of closed constant mean curvature $n-1$ surfaces separating the boundary from infinity.
\end{remark}

\section{Appendix}

The following lemma is used in the proof of \ref{family}.  This result shows that the conformally related metrics generated in the proof are asymptotically hyperbolic and have strictly smaller mass.  It defines explicitly the change of variable required to compute the mass. 
\begin{lemma}\label{Lemma:asymptotics}
	Let $(M^n,g)$ be an asymptotically hyperbolic manifold with $R_g \ge - n(n-1)$.  For $0 < s < 1$, consider a function $u_s$ whose expansion at infinity given by $u_s(t,x) = 1 + s\, v_n(x)t^n + o(t^n)$, where $v_n < 0$ is a function on $\mathbb{S}^{n-1}$.  The family of metrics $h_s = u_s^{\frac{4}{n-2}}g$ are asymptotically hyperbolic with $\mu(h_s) < \mu(g)$.
\end{lemma}
\begin{proof} 
	First, we use the change of variable have $t = \tau - \dfrac{2s\,v_n}{n(n-2)}\,\tau^{n+1}$.  Calculating the Taylor expansion at $t = \tau$, we have the following:
	\begin{align*}
		\sinh^{-2}(t) &= \sinh^{-2}(\tau)\Big(1 + \dfrac{4s\,v_n}{n(n-2)}\,\tau^n + \mathcal{O}(\tau^{n+2})\Big) \\
		u_s^{\frac{4}{n-2}}(t,x) &= 1 + \dfrac{4s\,v_n}{n-2}\,\tau^n + o(\tau^{n}) \\
		dt^2 &= \Big( 1 - \dfrac{4(n+1)s\,v_n}{n(n-2)}\,\tau^n + \mathcal{O}(\tau^{2n})\Big)d\tau^2\\
		t^n &= \tau^n + \mathcal{O}(\tau^{2n}).
	\end{align*}
	After substituting into the metric expansion as in Definition \ref{def:AH}, the metric $h_s$ takes the form
	\[ h_s = \dfrac{1}{\sinh^2(\tau)}\Big(d\tau^2 + \mathring{h} + \tau^n\Big( \dfrac{4s(n+1)\,v_n}{n(n-2)}\,\mathring{h} + \gamma\Big)+ \mathcal{O}(\tau^{n+1})\Big).\]
	As in \cite[Lemma 3.10]{Andersson2007}, we may use a change of coordinates (where we again use $\tau$ as the coordinate) that will put $h_s$ in Gaussian coordinates based on $\mathbb{S}^{n-1}$.  That is, we have
	\[ h_s= \frac{1}{\sinh^2(\tau)}\left(d\tau^2 + \mathring{h} + \tau^n\Big( \dfrac{4(n+1)s\,v_{n}}{n(n-2)}\,\mathring{h} + \gamma\Big)\right) \]
	where $\gamma$ is a $\tau$-depended tensor field on $\mathbb{S}^{n-1}$ such that the restriction to $\mathbb{S}^{n-1}$ is preserved after coordinate change, and $\tau$ measures the distance with respect to the metric $\sinh^2(\tau)h_s$.  
	
	Lastly, the mass aspect function for $h_s$ becomes
	\[ \mu(h_s) = \mathrm{tr}_{\mathring{h}} \left( \Big(\dfrac{4(n+1)\,s\,v_n}{n(n-2)}\,\mathring{h}\Big) + \gamma \right) = \dfrac{4(n-1)(n+1)\,s\,v_{n}}{n} + \mu(g). \]
	Since $v_{n} < 0$ and $s > 0$, we get that $\mu(h_s) < \mu(g)$. 
\end{proof}

The following is a result of Corvino, which can be found in \cite[Proposition 2.1]{Corvino2017}, and appears here for convenience.  It allows us to bump up scalar curvature in a compact region assuming that there does not exist a static potential.  

\begin{proposition}\label{proposition:bump}
	Let $V \subset M$ be a precompact smooth domain.  Suppose $g_0$ is a Riemannian metric on $M$ for which $L_{g_0}^*$ has trivial kernel on $V$, and let $V_0$ be an open set compactly contained in $V$.  There is a $C>0$, an $\varepsilon_0 > 0$ and a neighborhood $\mathcal{V}$ of $g_0$ in $C^{4,\alpha}(\overline{V})$, so that for smooth metrics $g$ on $M$ with $g|_{T\bar{V}} \in \mathcal{V}$ and functions $S \in C_c^{\infty}(M)$ with $\mathrm{spt}(S) \subset \overline{V_0}$ and $\| S \|_{C^{0,\alpha}} < \varepsilon_0$, there is a smooth metric $g+h$ on $M$ so that $h$ is supported in $\overline{V}$ and satisfies a bound $\| h\|_{C^{2,\alpha}} \le C\|S\|_{C^{0,\alpha}}$, and with $R(g+h) = R(g) + S$.  In fact, for each non-negative integer $k$, there is $C_k > 0$ so that $\|h \|_{C^{k,\alpha}} \le C_k \|S\|_{C^{k-2,\alpha}}$.
\end{proposition}

\printbibliography{AH MME Staticity}

@Article{Andersson2007,
  author        = {Lars Andersson and Mingliang Cai and Gregory J. Galloway},
  title         = {Rigidity and Positivity of Mass for Asymptotically Hyperbolic Manifolds},
  year          = {2007},
  month         = mar,
  archiveprefix = {arXiv},
  doi           = {10.1007/s00023-007-0348-2},
  eprint        = {math/0703259},
  primaryclass  = {math.DG},
}

@Article{Huang2019,
  author        = {Lan-Hsuan Huang and Hyun Chul Jang and Daniel Martin},
  title         = {Mass rigidity for hyperbolic manifolds},
  year          = {2019},
  month         = apr,
  archiveprefix = {arXiv},
  doi           = {10.1007/s00220-019-03623-0},
  eprint        = {1904.12010},
  primaryclass  = {math.DG},
}

@Article{Corvino2017,
  author   = {Justin Corvino},
  journal  = {Nonlinear Analysis in Geometry and Applied Mathematics},
  title    = {A note on the Bartnik mass},
  year     = {2017},
  volume   = {1},
  url      = {http://hdl.handle.net/10385/2466},
}

@Article{Pacheco2018,
  author        = {Armando J. Cabrera Pacheco and Carla Cederbaum and Stephen McCormick},
  title         = {Asymptotically hyperbolic extensions and an analogue of the Bartnik mass},
  year          = {2018},
  month         = feb,
  archiveprefix = {arXiv},
  doi           = {10.1016/j.geomphys.2018.06.010},
  eprint        = {1802.03331},
  primaryclass  = {math.DG},
}

@Article{Corvino2000,
  author   = {Justin Corvino},
  journal  = {Communications in Mathematical Physics},
  title    = {Scalar Curvature Deformatioon and a Gluing Construction for the Einstein Constraint Equations},
  year     = {2000},
  pages    = {137 - 189},
  doi      = {https://doi.org/10.1007/PL00005533},
}

@Article{Chrusciel2018,
  author        = {Piotr T. Chruściel and Gregory J. Galloway and Luc Nguyen and Tim-Torben Paetz},
  journal       = {Class. Quantum Grav. 35 (2018) 115015},
  title         = {On the mass aspect function and positive energy theorems for asymptotically hyperbolic manifolds},
  year          = {2018},
  month         = jan,
  archiveprefix = {arXiv},
  doi           = {10.1088/1361-6382/aabed1},
  eprint        = {1801.03442},
  primaryclass  = {gr-qc},
}

@Article{Bartnik89,
  author    = {Bartnik, Robert},
  journal   = {Phys. Rev. Lett.},
  title     = {New definition of quasilocal mass},
  year      = {1989},
  month     = {May},
  pages     = {2346--2348},
  volume    = {62},
  doi       = {10.1103/PhysRevLett.62.2346},
  issue     = {20},
  numpages  = {0},
  publisher = {American Physical Society},
  url       = {https://link.aps.org/doi/10.1103/PhysRevLett.62.2346},
}

@Article{HuiskenIlmanen01,
  author    = {Gerhard Huisken and Tom Ilmanen},
  journal   = {Journal of Differential Geometry},
  title     = {{The Inverse Mean Curvature Flow and the Riemannian Penrose Inequality}},
  year      = {2001},
  number    = {3},
  pages     = {353 -- 437},
  volume    = {59},
  doi       = {10.4310/jdg/1090349447},
  publisher = {Lehigh University},
  url       = {https://doi.org/10.4310/jdg/1090349447},
}

@Article{Dahl2015,
  author        = {Dahl, Mattias and Sakovich, Anna},
  title         = {A density theorem for asymptotically hyperbolic initial data satisfying the dominant energy condition},
  year          = {2015},
  month         = feb,
  archiveprefix = {arXiv},
  copyright     = {arXiv.org perpetual, non-exclusive license},
  doi           = {10.48550/ARXIV.1502.07487},
  eprint        = {1502.07487},
  primaryclass  = {math.DG},
  publisher     = {arXiv},
}

@Article{schoensimon,
  author  = {Schoen, Richard and Simon, Leon},
  journal = {Communications on Pure and Applied Mathematics},
  title   = {Regularity of stable minimal hypersurfaces},
  year    = {1981},
  number  = {6},
  pages   = {741-797},
  volume  = {34},
  doi     = {https://doi.org/10.1002/cpa.3160340603},
  eprint  = {https://onlinelibrary.wiley.com/doi/pdf/10.1002/cpa.3160340603},
  url     = {https://onlinelibrary.wiley.com/doi/abs/10.1002/cpa.3160340603},
}

\end{document}